\documentclass[12pt]{article}

\usepackage[top=2.4cm,bottom=2.2cm,left=2.6cm,right=2cm]{geometry}
\usepackage{latexsym,epsfig,psfrag}
\usepackage{eepic,colordvi,amscd}

\usepackage{amsmath,amsthm,dsfont,amsfonts,amssymb,fancyhdr}
\usepackage{enumerate}
\usepackage[usenames,dvipsnames]{color}
\usepackage{bbm,soul,supertabular,longtable,verbatim,extarrows}
\usepackage{titlesec}
\usepackage{graphicx}
\usepackage{BOONDOX-cal}
\usepackage{cite}
\usepackage{tikz}
\usepackage{booktabs,multirow,makecell}
\usepackage{authblk}
\usepackage{color}
\usetikzlibrary[trees]

\newcommand{\gskip}{\vspace{12pt}}

\usepackage{lipsum}
\usepackage{mathrsfs}
\usepackage{indentfirst}
\usepackage[colorlinks=true, allcolors=blue]{hyperref}
\newtheorem{thm}{Theorem}[section]

\newtheorem{lem}[thm]{Lemma}

\theoremstyle{definition}

\theoremstyle{remark}

\numberwithin{equation}{subsection}

\usepackage{mathtools}

\newcommand{\Rmnum}[1]{\uppercase\expandafter{\romannumeral #1}}  

\begin{document}
\title{Comments on the infinitely divisibility of the Conway--Maxwell--Poisson distribution}

\author{ Greg Markowsky, Preet Patel\\
	{\tt gmarkowsky@gmail.com} ~~{\tt ppat0019@student.monash.edu} \\
	Department of Mathematics,  Monash University, Australia
}

\maketitle
\begin{abstract}
	In an elegant recent paper \cite{geng2022conway}, Geng and Xia settled the question of the infinite divisibility of the Conway--Maxwell--Poisson distribution, using in large part several results from complex analysis. In this note we show how these complex analytic methods can be circumvented, thereby giving a proof of their result which is completely elementary.
\end{abstract}


The Conway--Maxwell--Poisson distribution is a discrete probability distribution on the nonnegative integers which admits the probability mass function

\begin{equation} \label{pmf}
P(X=k) = \frac{1}{Z(\lambda,\nu)}\frac{\lambda^k}{(k!)^\nu}.
\end{equation}

Here $\lambda, \nu >0$, or $\nu=0$ and $0<\lambda<1$, and $Z(\lambda,\nu)$ is the correct constant so that $\sum_{k=0}^\infty P(X=k) = 1$. This distribution has been many uses in the modelling of count data and other applications; see \cite{daly2015conway, hilbe2014modeling, li2020some, sellers2012poisson, yip2021forecasting, pogany2016integral}. Recently, in \cite{geng2022conway}, the following result was proved. 

\begin{thm}
The distribution defined by (\ref{pmf}) is infinitely divisible if and only if $\nu = 0$ or $\nu=1$.
\end{thm}


To understand the importance of this result in connection with limit theory and the law of small numbers, the reader is referred to the discussion in \cite{geng2022conway}. The proof given there used complex analytic methods in several places. In this brief note we will show how these methods can be avoided if desired.

\gskip

Let us begin by assuming that $X$ has distribution $CMP(\lambda,\mu)$; that is, with probability mass function given by (\ref{pmf}). If we assume that it is infinitely divisible, then, as noted by the authors of \cite{geng2022conway}, it must be compound Poisson, and there is an i.i.d. sequence $Y_1, Y_2, \dots$ of r.v.'s such that

$$
X \stackrel{d}{=} \sum_{n=1}^N Y_n,
$$

where $N \sim Pois(\mu)$ $(\mu>0)$ is independent of the $Y$'s. If the probability generating function of $Y_1$ is denoted by $G(x)$, that is, 

\begin{equation} \label{genfunc}
    G(x) = \sum_{r=1}^\infty P(Y_1=r) x^r,
\end{equation}

then 

\begin{equation} \label{entire?}
    e^{\mu G(x)} = \sum_{k=0}^\infty \frac{(\lambda x)^k}{(k!)^\nu}.
\end{equation}

Now, an interesting feature of (\ref{entire?}) is that the radius of convergence of the right side is clearly infinite, while the left side is only guaranteed to exist for $|x|<1$. The authors of \cite{geng2022conway} required at this step the knowledge that the series (\ref{genfunc}) has an infinite radius of convergence, and for this they turned to complex analysis, regarding $x$ as a complex variable. The argument required substantial subtlety, as for general power series we may have intricate cancellation between the terms; however, matters may be simplified here by noting that all series in question have non-negative coefficients, precluding difficulties due to cancellation, and the complex arguments may be replaced by the following lemma.

\begin{lem}
Suppose $(q_n)_n$ is a sequence of non-negative reals with $\sum_{n=0}^\infty q_n < \infty$. Define
\begin{align} \label{defG}
    G(x)=\sum_{n=0}^{\infty} q_n x^n
\end{align}
and let $F(x)=e^{\mu G(x)}$ for some $\mu>0$, defined wherever the series converges. Then $F$ is analytic, and has a power series expansion $F(x) = \sum_{n=0}^{\infty} t_n x^n$ near zero. Suppose this series for $F$ has an infinite radius of convergence. Then the series defining $G$ also has an infinite radius of convergence.
\end{lem}

\begin{proof}
The $\mu$ can be absorbed into $G$, and the coefficients remain non-negative, so we disregard $\mu$. Suppose (\ref{defG}) has a finite radius of convergence, $R$ (note that $R\geq 1$ since the $q_n$'s are summable). For $0<x<R$, we have
\begin{align} \label{rearrange}
\sum_{n=0}^{\infty} q_n x^n = G(x) < e^{G(x)}& =\sum_{k=0}^{\infty} \frac{\left(\sum_{n=0}^{\infty} q_n x^n \right)^k}{k!}
\end{align}

Since all terms in this double series are positive and the series converges absolutely and uniformly near 0, we may expand and rearrange without being cautious, and we recall that the series on the right must rearrange to $\sum_{k=0}^{\infty} t_k x^k$. We conclude that each $t_m$ can be expressed as a linear combination of the $q_n$'s with positive coefficients, and this linear combination must contain the term $q_m$ from the $k=1$ term in (\ref{rearrange}). We conclude that $q_m \leq t_m$, and this is enough to show that $R= \infty$. 
\end{proof}

The authors of \cite{geng2022conway} then proceeded through a series of probabilistic arguments in order to show that $|G(x)| \leq C|x|^d$ for some $C, d > 0 $ ($d$ is an integer) and $|x|$ sufficiently large. As viewed in the complex plane, $G$ is now known to be entire, and it is a standard exercise in complex analysis classes (generalizing Liouville's Theorem) to show that this implies that $G$ is a polynomial of degree $d$. This result can be deduced from Cauchy's Integral Formula or from a number of other closely related results. However, again in our case the fact that all coefficients are non-negative allows this argument to be replaced by the following lemma.

\begin{lem}
    Suppose $G(x)=\sum_{n=0}^{\infty} q_n x^n$, is a power series with $q_n \geq 0$ for all $n$. Suppose also that this series has infinite radius of convergence and satisfies $|G(x)| \leq C|x|^d$ for some $C, d > 0$ and $|x|$ sufficiently large. Then $G$ is a polynomial of degree at most $d$.
\end{lem}

\begin{proof}
    Suppose there is some nonzero $q_m$ with $m>d$. For $x \geq 0$ we have $G(x) \geq q_m x^m$, however as $m > d$ this quantity will dominate $Cx^n$ for any constant $C>0$, $n \leq d$ and $x$ sufficiently large. The lemma follows from this.
\end{proof}

\section*{Acknowledgements}

The authors would like to think Xi Geng and Aihua Xia for helpful conversations.

\bibliographystyle{plain}
\bibliography{123}

\begin{thebibliography}{1}

\bibitem{daly2015conway}
F.~Daly and R.~Gaunt.
\newblock The {C}onway-{M}axwell-{P}oisson distribution: distributional theory
  and approximation.
\newblock {\em ALEA}, 13:635–658, 2016.

\bibitem{geng2022conway}
X.~Geng and A.~Xia.
\newblock When is the {C}onway--{M}axwell--{P}oisson distribution infinitely
  divisible?
\newblock {\em Statistics \& Probability Letters}, 181:109264, 2022.

\bibitem{hilbe2014modeling}
J.~Hilbe.
\newblock {\em Modeling count data}.
\newblock Cambridge University Press, 2014.

\bibitem{li2020some}
B.~Li, H.~Zhang, and Jiao He.
\newblock Some characterizations and properties of {COM}-{P}oisson random
  variables.
\newblock {\em Communications in Statistics-Theory and Methods},
  49(6):1311--1329, 2020.

\bibitem{pogany2016integral}
T.~Pog{\'a}ny.
\newblock Integral form of the {COM}-{P}oisson normalization constant.
\newblock {\em Statistics \& Probability Letters}, 119:144--145, 2016.

\bibitem{sellers2012poisson}
K.~Sellers, D.~Borle, and G.~Shmueli.
\newblock The {COM}-{P}oisson model for count data: a survey of methods and
  applications.
\newblock {\em Applied Stochastic Models in Business and Industry},
  28(2):104--116, 2012.

\bibitem{yip2021forecasting}
S.~Yip, Y.~Zou, R.~Hung, and K.~Yiu.
\newblock Forecasting number of corner kicks taken in association football
  using overdispersed distribution.
\newblock {\em arXiv preprint arXiv:2112.13001}, 2021.

\end{thebibliography}
\end{document}